\documentclass[12pt]{article}

\usepackage{amsmath, amssymb, amsthm, amsfonts}
\usepackage{graphicx, xcolor}
\usepackage{enumerate}
\usepackage[all,cmtip]{xy}
\usepackage[a4paper, margin=1in]{geometry} % Adjust page layout
\usepackage{hyperref} % For hyperlinks
\usepackage{mathtools} % Extended math features
\usepackage{mathrsfs} % For script fonts
\usepackage{bm} % Bold math symbols

\theoremstyle{plain}
\newtheorem{theorem}{Theorem}[section]
\newtheorem{lemma}[theorem]{Lemma}

\theoremstyle{definition}

\theoremstyle{remark}
\newtheorem{remark}[theorem]{Remark}

\newcommand{\C}{\mathbb{C}}

\title{Abelian objects in categories with normal projections}
\author{Michael Hoefnagel \and Zurab Janelidze}
\date{}
\begin{document}

\maketitle
\begin{abstract} It is known that in (regular) unital and in subtractive categories, internal abelian groups are simply behaved; e.g., they are the same as internal algebras $(A,s)$ satisfying $s(x,0)=x$ and $s(x,x)=0$, i.e., \emph{subtraction algebras}. Moreover, in these categorical settings, such internal abelian group structures are unique, and every morphism between the underlying objects of internal abelian groups is necessarily a morphism of internal abelian groups. It is also known that both (regular) unital and subtractive categories have normal projections, i.e., the isomorphism formula $(X\times Y)/Y\approx X$ holds. In this paper, we show that all properties of simple behaviour of internal abelian groups in unital and subtractive categories lift to arbitrary categories having normal projections.
\end{abstract}

\begin{description}
    \item[MSC2020:] 18E13, 18A30, 08B25, 08B05
    \item[Keywords:] abelian object, centralic category, cokernel, congruence hyperextensible category, Gumm category, isomorphism formula, normal projection, Shifting lemma, subtraction, subtractive category, unital category 
\end{description}
\section*{Introduction}
If $\C$ is a pointed category with finite products, then for any two objects $X$ and $Y$ in $\C$, the canonical product inclusion $(1_X, 0)\colon X \to X \times Y$ is a kernel of the canonical product projection $\pi_2\colon X\times Y \to Y$, but it is not generally true that $\pi_2$ is a cokernel of $( 1_X, 0 )$. When this is the case, we say, following \cite{Janelidze2003}, that $\C$ has \emph{normal projections}. We may think of categories having normal projections as those pointed categories with finite products in which the following law holds:
\[
(X\times Y)/X\approx Y.
\]
The syntactic characterisation of algebraic categories having normal projections, obtained in \cite{Janelidze2003}, reveals that the property of normality of projections is widely present across a wide variety of algebraic categories of mathematical structures. This, on the one hand, makes the property fundamental, but on the other hand, it seems to suggest that, perhaps, the property is too weak to allow for significant consequences. In fact, up until now no significant result has been proven in the context of a category having normal projections. The interest in this property has rather been in that it is a property shared by two, in some sense complimentary, classes of pointed categories: unital categories \cite{Bourn1996} and subtractive categories \cite{ZJanelidze2005}. There is another common feature of these two classes of categories, which we recall below.

%is also a drawback, since any property implied by normality of projections will be common to a too wide variety of different categorical settings. The aim of this paper is to show that every category with normal projections naturally possesses notion of abelian object, the full subcategory of which is additive. We obtain this result as a consequence of our main characterisation theorem. 

A morphism \(s\colon X \times X \to X\) in a pointed category $\mathbb{C}$ is said to be an \emph{internal subtraction} if it satisfies
\[
s(x,x) = 0 \quad \text{and} \quad s(x,0) = x.
\]
Subtractive algebraic categories are precisely those whose forgetful functor to the category of pointed sets admits an internal subtraction (i.e., the corresponding Lawvere theory contains an internal subtraction, or equivalently, in the universal-algebraic language, the corresponding algebraic theory contains a binary term $s$ satisfying the identities above, where $0$ is the unique constant of the theory). It is known that when $\mathbb{C}$ is either a regular unital or a subtractive category, the following hold:
\begin{itemize}
\item[(A1)] Every morphism between objects that admit internal subtractions is homomorphic with respect to the subtractions.

\item[(A2)] Consequently, an internal subtraction $s\colon X\times X\to X$ is unique, when it exists, and it is the subtraction of an internal abelian group structure on $X$.

\item[(A3)] Combining the two results above, we also get that every morphism between two objects that admit an internal abelian group structure is a homomorphism of those (internal) abelian groups.
\end{itemize}
That the above holds for subtractive categories was established in \cite{BournJanelidze2009}. For regular unital categories, it is established in \cite{Bourn2022}. In fact, the properties above are established in \cite{Bourn2022} for a wider class of categories than regular unital categories: \emph{congruence hyperextensible categories}. In this paper, we show that the properties above are common to \emph{all} categories having normal projections, which include not only all unital and all subtractive categories, but as we note in this paper, also the congruence hyperextensible categories. In the language of \cite{Bourn2022}, we obtain that categories having normal projections are \emph{crystallographic for abelian groups}.

\section{Reformulations of the law \((A\times B)/B \cong A\)}

Before stating and proving the main result of this paper, we establish some characterisations of categories having normal projections. One of these characterisations will be useful for the main result, which is formulated and proved in the next section.

The proof of the following lemma is standard and left to the reader.

\begin{lemma}\label{LemC}
Given a split epimorphism \(r\colon R\to S\) with a right inverse \(s\colon S\to R\) and a morphism \(f\colon R\to T\), if 
\(f = ur\)
for some morphism \(u\), then 
\(
u = fs.
\)
Moreover, such a \(u\) exists if and only if 
$
f = fsr.$
\end{lemma}

By the above lemma, the universal property of a product projection \(\pi_2\colon A\times B\to B\) being a cokernel becomes equivalent to the implication
\begin{equation}\label{EqnA}
f\circ\iota_1=0\quad\Rightarrow\quad f=f\circ \iota_2\circ \pi_2,
\end{equation}
which must hold for any morphism \(f\colon A\times B\to A\).

\begin{theorem}\label{ThmA}
For a pointed category with binary products, the following conditions are equivalent:
\begin{enumerate}[(a)]
    \item The law \((A\times B)/B \cong A\) holds, i.e., (\ref{EqnA}) holds for all $f\colon A\times B\to C$. 
    \item The law \((X\times X)/X \cong X\) holds, i.e., (\ref{EqnA}) holds for all $f\colon X\times X\to Y$. 
    \item The following holds for all $f\colon X\times X\to Y$.
    \begin{equation}\label{EqnB}
f(1_X,0)=0\quad\Rightarrow\quad f(1_X,1_X)=f(0,1_X),
\end{equation}
    \item For all \(f\colon A\times B\to C\), \(a\colon X\to A\) and \(b\colon X\to B\), 
    \begin{equation}\label{EqnC}
    f(a,0)=0\quad\Rightarrow \quad f(a,b)=f(0,b).
    \end{equation}
    \item For all \(f\colon X\times X\to Y\) and \(x\colon U\to X\),
    \[
    f(x,0)=0\quad\Rightarrow \quad f(x,x)=f(0,x).
    \]
\end{enumerate}
\end{theorem}

\begin{proof}
(a)\(\Rightarrow\)(b) is trivial, while (b)\(\Rightarrow\)(c) is easy because of the following: \[f=f\circ \iota_2\circ \pi_2\quad\Rightarrow\quad f(1_X,1_X)=f\circ \iota_2\circ \pi_2\circ (1_X,1_X)=f(0,1_X).\]
(d)$\Rightarrow$(e) is also trivial. We can get (e)$\Rightarrow$(b) by setting $x=1_X$ in (e) and we can get (d)$\Rightarrow$(a) by setting $a=\pi_1,b=\pi_2$ in (d). To complete the proof, it suffices to show (c)\(\Rightarrow\)(d). Assume (b). Suppose \(f(a,0)=0\). Then
\(
f\circ (a\times b)\circ \iota_1=0.
\)
Hence,
\(
f(a,b)=f\circ (a\times b)\circ (1_X,1_X)=f\circ (a\times b)\circ (0,1_X)=f(0,b).
\)
\end{proof}

\begin{remark}
The equivalence of (a) and (b) in Theorem~\ref{ThmA} was first established in \cite{Janelidze2004} through a direct proof of (b)\(\Rightarrow\)(a) that applies to the diagram
\[
\xymatrix{
A\times B\ar[r]^-{\iota_1}\ar[d]_-{\pi_1} & (A\times B)\times(A\times B)\ar[r]^-{\pi_2}\ar@/_3pt/[d]_-{\pi_1\times \pi_2} & A\times B \\
A\ar[r]_-{\iota_1} & A\times B\ar[r]_-{\pi_2}\ar@/_3pt/[u]_-{\iota_1\times\iota_2} & B\ar[u]_-{\iota_2}
}
\]
the following general fact: in a pointed category, consider a diagram
\[
\xymatrix{
\bullet\ar[r]\ar[d]\ar@{}[dr]|-{(1)} & \bullet\ar@{}[dr]|-{(2)}\ar[r]\ar@/_3pt/[d] & \bullet \\
\bullet\ar[r] & \bullet\ar[r]\ar@/_3pt/[u] & \bullet\ar[u]
}
\]
where (1) and (2) commute, and the middle downward morphism is a split epimorphism with the middle upward morphism its right inverse. If the top right arrow is a cokernel of the top left arrow and the bottom right arrow is an epimorphism, then it is the cokernel of the bottom left arrow.
\end{remark}

\section{Abelian objects}

In what follows, we make use of the language of generalised elements. A morphism \(s\colon X \times X \to X\) is called a \emph{subtraction} on $X$ if the following laws hold:
\[
s(x,x) = 0 \quad \text{and} \quad s(x,0) = x.
\] Given a subtraction $s$ on $X$ and a subtraction $s'$ on $X'$, we say that a morphism $f\colon X\to X'$ is \emph{homomorphic} with respect to these subtractions, if the following law holds:
\[
f(s(x,y))=s'(f(x),f(y)).
\]
As we can see from \cite{BournJanelidze2009}, if (A1) from the Introduction holds in a pointed category with finite products, then so do (A2) and (A3). Let us call a pointed category having finite products and satisfying (A1-3) a category having a \emph{subtractive theory of abelian objects}. By an \emph{abelian object} me mean an object that admits an internal abelian group structure (and such structure is unique).

\begin{theorem}
If a category has normal projections then it has a subtractive theory of abelian objects.
\end{theorem}

\begin{proof}
First, we show that in a category having normal projections, every subtraction $s$ satisfies the law
$$s(s(x,z),s(y,z))=s(x,y).$$
This law is known to be equivalent to requiring that $s$ is the subtraction of a group (see, e.g., \cite{ZJanelidze2006c}). We can establish it by defining
\[f(z,(x,y))=s(s(x,z),s(y,z)).\]
Clearly, $f(z,(0,0))=0$. By Theorem~\ref{ThmA}, we get: $f(z,(x,y))=f(0,(x,y))$, i.e.,
$$s(s(x,z),s(y,z))=s(s(x,0),s(y,0))=s(x,y).$$
Now, let \(g\colon X \rightarrow X'\) and consider subtractions $s$ on $X$ and $s'$ on $X'$. By what we showed above, both $s$ and $s'$ are subtractions of groups. Let $a$ and $a'$ denote the corresponding additions, $$a(x,y)=s(x,s(0,y))\quad a'(x,y)=s'(x,s'(0,y)).$$ Define
\[
f(x,y) = s'(g(a(x,y)),a'(g(x),g(y))).
\]
Then \(f(x,0) = 0\), and so $f(x,y)=f(0,y)$ by Theorem~\ref{ThmA}. This gives
\[
s'(g(a(x,y)),a'(g(x),g(y)))=s'(g(a(0,y)),a'(g(0),g(y)))=s'(g(y),g(y))=0,
\]
which implies that $g(a(x,y))=a'(g(x),g(y)).$ This proves that $g$ is a homomorphism of groups, and hence homomorphic with respect to $s$ and $s'$.
\end{proof}

The construction of $f$ at the start of the proof of the theorem above, which is a crucial step in the proof, actually imitates a construction in the proof of a similar result for congruence hyperextensible categories found in \cite{Bourn2022}. The reason why such imitation is possible is that such categories actually have normal projections. Congruence hyperextensibility is a weakening of Gumm's shifting lemma. The implication (\ref{EqnC}) can be described as a further weakening of the same shifting lemma. Indeed, in a pointed category with pullbacks, writing $\mathsf{Ker}(f)$ for the kernel congruence of $f$, the implication (\ref{EqnC}) can be rewritten as the implication
$$(x,0)\mathsf{Ker}(f)(0,0)\quad\Rightarrow\quad (x,y)\mathsf{Ker}(f)(0,y).$$
This implication fits precisely the shape of the shifting lemma, as shown below.
\[\xymatrix@=50pt{ 
    (x,0) \ar@{-} \ar@{-}[r]_-{\mathsf{Ker}(\pi_1)} \ar@{-}[d]^-{\mathsf{Ker}(\pi_2)} \ar@{-}@/_10pt/[d]_-{\mathsf{Ker}(f)} & (x,y) \ar@{-}[d]_-{\mathsf{Ker}(\pi_2)} \ar@{..}@/^10pt/[d]^-{\mathsf{Ker}(f)} \\
    (0,0) \ar@{-}[r]^-{\mathsf{Ker}(\pi_1)} & (0,y) }\]
It is also an instance of the shape of the weaker form of the shifting lemma that defines a congruence hyperextensible category. Thus, we have the following result.

\begin{theorem}
Every congruence hyperextensible category has normal projections. 
\end{theorem}

The result implies that pointed congruence modular varieties, and more generally, pointed Gumm categories \cite{BournGran2004}, have normal projections. In fact, by the characterisation of Gumm categories in terms of congruence hyperextensible categories given in \cite{Bourn2022}, Gumm categories have normal local projections \cite{Janelidze2004} --- as we already know this from \cite{Hoefnagel2019b}.

The theorem above is also a simple consequence of Remark~2.11 in \cite{Hoefnagel2023}, which notes that every congruence hyperextensible category is \emph{centralic}, i.e., the following law holds:
$$f(x,0)=f(y,0)\quad\Rightarrow\quad f(x,z)=f(y,z).$$
Centralic categories have normal projects and include unital, weakly unital \cite{Martins-Ferreira-PHD}, pointed majority \cite{Hoefnagel2018a}, and pointed factor-permutable categories \cite{Gran2004} (but do not include subtractive categories).

%\bibliographystyle{plain}
%\bibliography{references}
\end{document}